\documentclass[reqno, 11pt, letterpaper]{amsart}

\usepackage{amsmath}
\usepackage{amsfonts}
\usepackage{amssymb}
\usepackage{graphicx}
\usepackage{amsthm,color,yfonts,cite}
\usepackage{paralist}
\usepackage{hyperref}
\usepackage{physics}
\usepackage{bbm}
 \usepackage{bm}
\usepackage{comment}

\newtheorem{theorem}{Theorem}

\newtheorem{proposition}[theorem]{Proposition}
\newtheorem{lemma}[theorem]{Lemma}
\newtheorem{corollary}[theorem]{Corollary}

\theoremstyle{remark}
\newtheorem{remark}[theorem]{Remark}

\newcommand{\ls}{\lesssim}

\newcommand{\la}{\langle}
\newcommand{\ra}{\rangle}
\newcommand{\R}{\mathbb{R}}

\newcommand{\Z}{\mathbb{Z}}
\newcommand{\pa}{\partial}
\newcommand{\ep}{\epsilon}

\usepackage{wrapfig}
\usepackage{tikz}
\usetikzlibrary{arrows,calc,decorations.pathreplacing}
\definecolor{light-gray1}{gray}{0.90}
\definecolor{light-gray2}{gray}{0.80}
\definecolor{light-gray3}{gray}{0.60}

\numberwithin{equation}{section}

\numberwithin{theorem}{section}

\numberwithin{table}{section}

\numberwithin{figure}{section}

\ifx\pdfoutput\undefined
  \DeclareGraphicsExtensions{.pstex, .eps}
\else
  \ifx\pdfoutput\relax
    \DeclareGraphicsExtensions{.pstex, .eps}
  \else
    \ifnum\pdfoutput>0
      \DeclareGraphicsExtensions{.pdf}
    \else
      \DeclareGraphicsExtensions{.pstex, .eps}
    \fi
  \fi
\fi

\title[Long-time reduction for the biharmonic NLS]{Long-time reduction for the biharmonic nonlinear Schr\"odinger equation}

\date{\today}
\author[Y. Hong]{Younghun Hong}
\address{Department of Mathematics, Chung-Ang University, Seoul 06974, South Korea}
\email{yhhong@cau.ac.kr}

\author[J. Jang]{Junyeong Jang}
\address{Department of Mathematics, Chung-Ang University, Seoul 06974, South Korea}
\email{jyjang0119@cau.ac.kr}

\begin{document}

\begin{abstract}

In this article, we study the reduction of the one-dimensional biharmonic nonlinear Schrödinger equation to the cubic nonlinear Schrödinger equation in the vanishing higher-order dispersion limit. In the intermediate regularity regime $0<s<2$, where the energy conservation law does not control the $H^s$--norm, we prove the long-time $L^2$--convergence of $H^s$--solutions with an exponential-in-time approximation bound. The argument relies mainly on persistence of regularity. This result provides a simple example of its use in limit problems without relying on higher-order conservation laws.
\end{abstract}

\maketitle


\section{Introduction}

We consider the Cauchy problem for the one-dimensional biharmonic nonlinear Schr\"odinger equation (NLS) with a cubic nonlinearity:
\begin{equation}\label{eq: biharmonic NLS}
\left\{
\begin{aligned}
i\partial_t u_\epsilon &= (-\partial_x^2 + \epsilon^2 \partial_x^4) u_\epsilon + \lambda |u_\epsilon|^2 u_\epsilon,\\
u_\epsilon(0) &= u_{\epsilon,0} \in H^s(\mathbb{R}),
\end{aligned}
\right.
\end{equation}
where $u_\epsilon = u_\epsilon(t,x): I (\subset \mathbb{R}) \times \mathbb{R} \to \mathbb{C}$ and $\epsilon > 0$ is a small parameter. By a suitable scaling and normalization, we may assume without loss of generality that the coupling constant satisfies $\lambda = 1$ (defocusing case) or $\lambda = -1$ (focusing case).
The biharmonic NLS arises as a higher-order dispersive correction to the classical NLS when second-order dispersion is insufficient. It is derived by retaining fourth-order terms in the expansion of the dispersion relation, leading to the appearance of the bi-Laplacian operator. Such models are relevant in physical settings where lower-order dispersion is weak or vanishes, including certain optical media and Bose--Einstein condensates. Consequently, the equation captures finer-scale dispersive effects and stronger regularization at high frequencies \cite{FIP2002, Kar1996, KS2000}.

By standard arguments, the equation \eqref{eq: biharmonic NLS} is locally well-posed in $H^s(\mathbb{R})$ for $s \ge 0$ (see, for instance, Ben-Artzi-Koch-Saut \cite{BKS2000} and Cazenave \cite{Cazenave2003}). 
The corresponding solution conserves the mass
\begin{equation}\label{eq: L2 cons law}
\mathcal{M}[u] := \|u\|_{L^2(\mathbb{R})}^2, 
\end{equation}
and if $s\geq 2$, also the energy
$$
\mathcal{E}_\epsilon[u] := \frac{1}{2}\|\partial_x u\|_{L^2(\mathbb{R})}^2 
+ \frac{\epsilon^2}{2}\|\partial_x^2 u\|_{L^2(\mathbb{R})}^2 
+ \frac{\lambda}{4}\|u\|_{L^4(\mathbb{R})}^4.
$$
Since the equation is $L^2$-subcritical, these conservation laws immediately yield global well-posedness in $L^2(\mathbb{R})$ and in $H^2(\mathbb{R})$. Moreover, by the persistence of regularity, one can show global well-posedness\footnote{By global $H^s(\mathbb{R})$-solutions, we mean solutions such that $u_\epsilon(t), u(t) \in C_t([-T,T]; H_x^s(\mathbb{R}))$ for any $T > 0$. However, the norms $\|u_\epsilon(t)\|_{H_x^s(\mathbb{R})}$ and $\|u(t)\|_{H_x^s(\mathbb{R})}$ may grow in time.} in $H^s(\mathbb{R})$ for $s>0$ (see Corollary \ref{cor: Hs norm growth}).
Concerning dynamical properties, we also refer to Pausader \cite{Pau2007, Pau2009-1, Pau2009-2} and Boulenger-Lenzmann \cite{BL2017} for related higher-dimensional results on scattering and blow-up for \eqref{eq: biharmonic NLS}.

In this article, we study the reduction of \eqref{eq: biharmonic NLS} to the standard nonlinear Schr\"odinger equation
\begin{equation}\label{eq: NLS}
\left\{\begin{aligned}
i\partial_t u &= -\partial_x^2 u + \lambda |u|^2 u,\\
u(0) &= u_0 \in H^s(\mathbb{R}).
\end{aligned}\right.
\end{equation}
in the limit $\epsilon \to 0$. In fact, it is not difficult to show that if $u_{\epsilon,0} \to u_0$ in $H^s(\mathbb{R})$ with $s \ge 0$, then there exists $T>0$ sufficiently small such that
$$
\|u_\epsilon(t) - u(t)\|_{C_t([-T,T];L_x^2(\mathbb{R}))} \to 0,
$$
where $u_\epsilon(t)$ and $u(t)$ denote the solutions to \eqref{eq: biharmonic NLS} and \eqref{eq: NLS}, respectively. Moreover, in the case $s = 2$, one can combine the mass and energy conservation laws to extend the interval of convergence arbitrarily, similarly to \cite{HJY2025}. However, when $0 < s < 2$, such a long-time extension is not immediate due to the lack of energy conservation at this level of regularity.

Our main result establishes long-time convergence in this intermediate regularity regime, where the persistence of regularity plays a crucial role.
\begin{theorem}[Long-time reduction for biharmonic NLS]\label{thm: main theorem}
For $0 < s < 2$, suppose that $u_0\in H_x^s$ and 
$$
\sup_{\ep\in(0,1]}\|u_{\ep,0}\|_{H_x^s}<\infty,
$$
and let $u_\epsilon(t)$ (resp., $u(t)$) be the global $H^s(\mathbb{R})$-solution to the biharmonic NLS \eqref{eq: biharmonic NLS} (resp., the NLS \eqref{eq: NLS}) with initial data $u_{\ep,0}$ (resp., $u_0$). Then, there exist constants $C_0, C_s>0$, independent of $\ep\in(0,1]$ and $t\in\R$, such that for all $t \in \mathbb{R}$, 
\begin{equation}\label{eq: convergence bound}
\|u_\epsilon(t) - u(t)\|_{L_x^2(\mathbb{R})}
\leq
C_s\Big(\epsilon^{\frac{s}{2}} + \|u_{\epsilon,0} - u_0\|_{L^2(\mathbb{R})}\Big) e^{C_0|t|}.
\end{equation}
\end{theorem}

\begin{remark}
(i) If two solutions evolve from the same initial data $u_0$, then the convergence bound \eqref{eq: convergence bound} implies that they remain close up to times of order $\log \frac{1}{\epsilon}$.\\
(ii) To be precise, the exponential rate $C_0$ in Theorem \ref{thm: main theorem} can be chosen depending only on $a$ and the lower-regularity bounds $\|u_0\|_{L_x^2}$ and $\displaystyle\sup_{\ep\in(0,1]}\|u_{\ep,0}\|_{L_x^2}$, whereas $C_s$ may additionally depend on the higher-regularity bounds $\|u_0\|_{H_x^s}$ and $\displaystyle\sup_{\ep\in(0,1]}\|u_{\ep,0}\|_{H_x^s}$.
\end{remark}

\begin{remark}[Conservation laws and long-time iteration]
(i) In settings where conservation laws are available (see, for instance, \cite{HY2019, HJY2025}), one can obtain uniform-in-time bounds for the solutions. Consequently, long-time convergence with an exponential-in-time bound of the form \eqref{eq: convergence bound} follows by iterating the local-in-time convergence result. \\
(ii) On the other hand, in the absence of suitable conservation laws, as in the present article, persistence of regularity may yield only exponentially growing bounds for the solutions. To establish global-in-time convergence with an exponential-in-time bound, one might expect that such weak control needs to be improved by employing the \textit{I-method} or the \textit{upside-down I-method} \cite{Vedran2011-1,Vedran2011-2,Vedran2012, CKSTT2001, CKSTT2003-1, CKSTT2003-2}. 
\end{remark}
\begin{remark}[Persistence of regularity]
(i) A key observation of this work, however, is that such a refinement through delicate analysis is unnecessary. In fact, a much simpler argument based on persistence of regularity suffices to obtain the desired result. 
Indeed, the persistence of regularity allows the $H^s$ regularity to be propagated on successive time intervals whose length is determined solely by the conserved $L^2$ norm. Although this argument yields an exponential-in-time $H^s$ bound, it is still sufficient to iterate the local-in-time convergence result and prove Theorem~\ref{thm: main theorem}. The exponentially growing bounds for the solutions are less harmful than one might expect, as the differences are not estimated solely via Gr\"onwall's inequality. The only cost of allowing growing solution bounds is an increase in the constant $C_0$ in the convergence estimate \eqref{eq: convergence bound}. 
Thus, unlike the arguments in \cite{HY2019, HJY2025}, which use conservation laws to obtain uniform-in-time bounds, the long-time reduction is obtained without either an $H^s$ conservation law or a modified energy method.\\
(ii) We expect that this approach extends more broadly. The reason for considering the biharmonic NLS here is that it provides a simple mathematical setting in which the main idea can be clearly illustrated without inessential technical complications, while still retaining physical relevance. In particular, one can avoid the use of local smoothing norms, since Strichartz estimates suffice. In future work, we plan to extend this observation to more physically relevant models, including the Boussinesq equation \cite{HY2024}, the Fermi--Pasta--Ulam system \cite{HKY2021, KY2025}, and possibly water wave (and related) problems \cite{SW2000, BCL2005}. These systems typically possess conservation laws that control low-regularity quantities such as the $L^2$-norm, but lack conservation laws governing higher Sobolev norms.
\end{remark}

\subsection{Outline of the paper}
The rest of this paper is organized as follows. In Section~\ref{sec: prelim}, we establish uniform Strichartz estimates and estimates for the difference of localized linear flows. In Section~\ref{sec: persistence of regularity LWP}, we prove uniform local well-posedness with persistence of regularity and derive exponential $H^s$--bounds. Finally, in Section~\ref{sec: main thm proof}, we combine the local-in-time convergence result with the persistence of regularity to prove our main theorem.

\subsection{Acknowledgement}
This work was supported by the National Research Foundation of Korea(NRF) grant funded by the Korean government (MSIT) (No. RS-2023-00219980 and RS-2026-25479401).

\section{Preliminaries}\label{sec: prelim}

For $\theta \in [0,1]$, we consider the Schr\"odinger flow
$$
S_\epsilon^\theta(t)
:=e^{it(\partial_x^2-\theta\epsilon^2 \partial_x^4)},
$$
which interpolates between the standard free Schr\"odinger flow $S(t):=S_\epsilon^0(t)=e^{it\partial_x^2}$ and the biharmonic free Schr\"odinger flow $S_\epsilon(t) := S_\epsilon^1(t)=e^{it(\partial_x^2-\epsilon^2 \partial_x^4)}$. For convenience, we denote 
\begin{equation}
s_\ep^\theta(\xi)
=\xi^2+\theta\ep^2\xi^4,\quad s(\xi)=\xi^2\quad\textup{and}\quad s_\ep(\xi)=\xi^2+\ep^2\xi^4,
\end{equation}
so that
$$S_\epsilon^\theta(t)=e^{-its_\epsilon^\theta(-i\partial_x)},\quad S(t)=e^{-its(-i\partial_x)}\quad\textup{and}\quad S_\epsilon(t)=e^{-its_\epsilon(-i\partial_x)}.$$
In this section, we summarize some basic properties of this interpolated Schr\"odinger flow.

\subsection{Strichartz estimates}
We say that $(q,r)$ is \textit{admissible} if $4< q\leq \infty$, $2\leq r< \infty$ and
$$\frac{2}{q}+\frac{1}{r}=\frac{1}{2}.$$

\begin{lemma}[Strichartz estimates]\label{lem: Strichartz}
Let $\epsilon\geq0$ and $\theta\in[0,1]$. For admissible pairs $(q,r)$ and $(\tilde{q},\tilde{r})$, we have
\begin{equation}\label{ineq: Strichartz 1}
\|S_\epsilon^\theta(t)u_0\|_{L_t^q(\mathbb{R}; L_x^r(\mathbb{R}))}\lesssim \|u_0\|_{L^2(\mathbb{R})},
\end{equation}
and
\begin{equation}\label{ineq: Strichartz 2}
\bigg\|\int_0^tS_\epsilon^\theta(t-t_1)F(t_1)dt_1\bigg\|_{L_t^q(\mathbb{R}; L_x^r(\mathbb{R}))}\lesssim \|F\|_{L_t^{\tilde{q}'}(\mathbb{R}; L_x^{\tilde{r}'}(\mathbb{R}))}.
\end{equation}
\end{lemma}

\begin{proof}[Sketch of the proof]
The linear flow admits the integral representation
$$
S_\epsilon^\theta(t)u_0(x) 
=
\frac{1}{2\pi}\int_{-\infty}^\infty\bigg\{\int_{-\infty}^\infty e^{-its_\ep^\theta(\xi)}e^{i(x-y)\xi}\, d\xi\bigg\} u_0(y)\, dy.
$$
Let $\chi\in C_c^\infty(0,\infty)$ be a smooth cutoff satisfying $\operatorname{supp}\chi \subset \big(\tfrac{1}{2},2\big)$ and $\sum_{k=-\infty}^\infty \chi\big(\frac{|\xi|}{2^k}\big) \equiv 1$. Then, the integral kernel admits the dyadic decomposition
$$
\begin{aligned}
\int_{-\infty}^\infty e^{-its_\ep^\theta(\xi)}e^{i(x-y)\xi}\, d\xi
&= \sum_{k=-\infty}^\infty \int_{-\infty}^\infty e^{-its_\ep^\theta(\xi)}e^{i(x-y)\xi}\chi\bigg(\frac{|\xi|}{2^k}\bigg)\, d\xi \\
&= \sum_{k=-\infty}^\infty 2^k\int_{-\infty}^\infty e^{-its_\ep^\theta(2^k\xi)}e^{i2^k(x-y)\xi}\chi(|\xi|) \, d\xi.
\end{aligned}
$$
Note that for each integral in the sum, the phase function $s_\ep^\theta(2^k\xi)$ is uniformly nondegenerate, since
$$
\Big(s_\ep^\theta(2^k\xi)\Big)'' = 2^{2k}\big(2+12\theta\epsilon^2 2^{2k}\xi^2\big)\geq2^{2k}.
$$
Consequently, by the Van der Corput lemma, we have
\begin{equation}\label{ineq: kernel estimate}
\bigg|2^k\int_{-\infty}^\infty e^{-its_\ep^\theta(2^k\xi)}e^{i2^k(x-y)\xi}\chi(|\xi|)\, d\xi\bigg|
\ls \frac{1}{|t|^{1/2}}.
\end{equation}
Let $P_k$ denote the Littlewood-Paley projection defined by
$$
\widehat{P_kf}(\xi):=\chi\bigg(\frac{|\xi|}{2^k}\bigg)\hat{f}(\xi).
$$
Then \eqref{ineq: kernel estimate} yields the dispersive estimate
$$
\big\|S_\ep^\theta(t)P_ku_0\big\|_{L_x^\infty}
\ls |t|^{-1/2}\|u_0\|_{L_x^1},
$$
where the implicit constant is independent of $k$, $\ep$ and $\theta$. Interpolating this estimate with the trivial inequality $\|S_\ep^\theta(t)P_ku_0\|_{L_x^2}\leq\|P_ku_0\|_{L_x^2}$, and applying the standard $TT^*$ argument in \cite{KT1998}, we obtain
$$
\|S_\ep^\theta(t)P_ku_0\|_{L_t^qL_x^r}\ls \|P_ku_0\|_{L_x^2},
$$
and
$$
\bigg\|\int_0^tS_\ep^\theta(t-t_1)P_kF(t_1)dt_1\bigg\|_{L_t^qL_x^r}\ls \|P_kF\|_{L_t^{\tilde{q}'}L_x^{\tilde{r}'}}.
$$
The implicit constants are independent of $k$, $\ep$ and $\theta$. 
Since $q,r \geq 2$ and $r<\infty$, the Littlewood–Paley inequality and Minkowski’s inequality imply
$$
\|S_\ep^\theta(t)u_0\|_{L_t^qL_x^r}
\ls
\bigg(\sum_{k\in\Z}\|S_\ep^\theta(t)P_ku_0\|_{L_t^qL_x^r}^2\bigg)^{\frac{1}{2}}
\ls
\bigg(\sum_{k\in\Z}\|P_ku_0\|_{L_x^2}^2\bigg)^{\frac{1}{2}}
\ls \|u_0\|_{L_x^2}.
$$
Similarly, since $\tilde{q}'\leq2$ and $1<\tilde{r}'\leq2$, we obtain
$$
\begin{aligned}
\bigg\|\int_0^tS_\ep^\theta(t-t_1)F(t_1)dt_1\bigg\|_{L_t^qL_x^r}
&\ls
\bigg(\sum_{k\in\Z}\bigg\|\int_0^tS_\ep^\theta(t-t_1)P_kF(t_1)dt_1\bigg\|_{L_t^qL_x^r}^2\bigg)^{\frac{1}{2}}\\
&\ls
\bigg(\sum_{k\in\Z}\|P_kF\|_{L_t^{\tilde{q}'}L_x^{\tilde{r}'}}^2\bigg)^{\frac{1}{2}}
\ls
\|F\|_{L_t^{\tilde{q}'}L_x^{\tilde{r}'}}.
\end{aligned}
$$
Here, the pair $(4,\infty)$ is excluded from the admissible range, since the Littlewood--Paley inequality fails in $L^\infty(\mathbb{R})$.
\end{proof}

\subsection{Convergence of the linear flow}
Next, we show the convergence of the linear flow $S_\epsilon(t)\to S(t)$ within the low frequency window $|\xi|\leq N_0:=\frac{1}{2}\ep^{-\frac{1}{2}}$. For the statement, let $P_{\leq N}$ be the sharp frequency cut-off defined by $\widehat{P_{\leq N}}f=\mathbbm{1}_{|\xi|\leq N}\hat{f}(\xi)$. For $T>0$ and $s\geq0$, we define the solution space by 
$$
X_T^s:=L_t^\infty([-T,T];H_x^s(\R))\cap L_t^8([-T,T];W_x^{s,4}(\R)),
$$
equipped with the norm (we denote $L_T^p:=L_t^p([-T,T])$)
$$
\|u\|_{X_T^s}:=\|u\|_{L_T^\infty H_x^s}+\|u\|_{L_T^8 W_x^{s,4}}.
$$

\begin{lemma}[Estimates for difference of localized linear flow]
Let $N_0:=\frac{1}{2}\ep^{-\frac{1}{2}}$. Then, for $0<T\leq 1$, $0<\ep\leq1$ and $0\leq s \leq 4$, we have
\begin{equation}\label{ineq: freq loc diff est 1}
\big\|\big(S_\ep(t)-S(t)\big)P_{\leq N_0}u_0\big\|_{X_T^0}
\ls
\ep^{\frac{s}{2}} T\|u_0\|_{H_x^s},
\end{equation}
and, if $(q,r)$ is admissible,  
\begin{equation}\label{ineq: freq loc diff est 2}
\bigg\|\int_0^t\big(S_\ep(t-t_1)-S(t-t_1)\big)(P_{\leq N_0}F)(t_1) \, dt_1
\bigg\|_{X_T^0}
\ls 
\ep^{\frac{s}{2}}T\|F\|_{L_T^{q'}W_x^{s,r'}}.
\end{equation}
\end{lemma}

\begin{proof}
Since $N_0=\frac{1}{2}\ep^{-\frac{1}{2}}$, we have
\begin{equation}\label{ineq: phase function difference}
|s_\ep(\xi)-s(\xi)|
=\ep^2|\xi|^4
\leq\ep^2N_0^{4-s}|\xi|^s
\ls \ep^{\frac{s}{2}}|\xi|^s\quad 
\textup{where} \quad |\xi|\leq N_0. 
\end{equation}
Moreover, by the fundamental theorem of calculus, we have
$$
e^{-its_\epsilon(\xi)}-e^{-its(\xi)}=\int_0^1 \frac{d}{d\theta}e^{-its_\epsilon^\theta(\xi)} d\theta
=
\int_0^1e^{-its_\epsilon^\theta(\xi)}(-it\ep^2\xi^4) \, d\theta.
$$
Thus, by the Minkowski inequality, \eqref{ineq: phase function difference} and Lemma~\ref{lem: Strichartz}, we prove 
$$
\begin{aligned}
\big\|\big(S_\ep(t)-S(t)\big)P_{\leq N_0}u_0\big\|_{X_T^0}
\leq
\int_0^1
\big\|tS_\epsilon^\theta(t)\ep^2\pa_x^4 P_{\leq N_0}u_0\big\|_{X_T^0} \, d\theta
\ls
\ep^{\frac{s}{2}} T\int_0^1\|u_0\|_{H_x^s} \, d\theta
\ls
\ep^{\frac{s}{2}} T\|u_0\|_{H_x^s}.
\end{aligned}
$$
Similarly, for \eqref{ineq: freq loc diff est 2}, since $|t-t_1|\leq T$ on $t\in[-T,T]$, we can show
$$
\begin{aligned}
&\bigg\|
\int_0^t(S_\ep(t-t_1)-S(t-t_1))(P_{\leq N_0}F)(t_1) \, dt_1
\bigg\|_{X_T^0}\\
&\leq
T\int_0^1
\bigg\|
\int_0^tS_\epsilon^\theta(t-t_1)\ep^2\pa_x^4(P_{\leq N_0}F)(t_1) \, dt_1
\bigg\|_{_{X_T^0}}\, d\theta \\
&\ls
T\|\ep^2\pa_x^4P_{\leq N_0}F\|_{L_T^{q'}L_x^{r'}}
\ls 
\ep^{\frac{s}{2}}T\|F\|_{L_T^{q'}W_x^{s,r'}}.
\end{aligned}
$$
\end{proof}

\section{Local well-posedness with persistence of regularity}\label{sec: persistence of regularity LWP}

In this section, we prove the local well-posedness of the biharmonic NLS \eqref{eq: biharmonic NLS}, with an emphasis on persistence of regularity and uniformity in $\epsilon \in (0,1]$. 
By persistence of regularity, we mean that the existence time in a higher-regularity local well-posedness result can be taken to be the same as that in a lower-regularity well-posedness result.

\begin{proposition}[Uniform local well-posedness for biharmonic NLS with persistence of regularity]\label{prop: persistence of regularity}
Let $s\geq 0$ and assume that
$$
\sup_{\ep\in(0,1]}\|u_{\ep,0}\|_{L_x^2}=: R_0,
\qquad
\sup_{\ep\in(0,1]}\|u_{\ep,0}\|_{H_x^s}=: R_s.
$$
Then, there exist $T=T(R_0)\sim R_0^{-4}>0$ and $C_s=C_s(R_0)>0$, independent of $\ep$ and $R_s$, and a unique solution $u_\ep\in C_t([-T,T];H_x^s(\R))\cap L_t^8([-T,T];W_x^{s,4}(\R))$ to \eqref{eq: biharmonic NLS} such that
$$
\|u_\ep\|_{L_T^\infty L_x^2}+\|u_\ep\|_{L_T^8 L_x^4}\leq C_0 R_0,
\qquad
\|u_\ep\|_{L_T^\infty H_x^s}+\|u_\ep\|_{L_T^8 W_x^{s,4}} \leq C_s R_s.
$$
\end{proposition}
\begin{remark}[Uniform local well-posedness for cubic NLS with persistence of regularity]\label{rmk: NLS LWP}
By the analogous proof, the corresponding statement with the same $T=T(R_0)$ and $C_s=C_s(R_0)$ holds for the cubic NLS \eqref{eq: NLS}. For more details, see \cite{Cazenave2003}.
\end{remark}

\begin{proof}[Proof of Proposition~\ref{prop: persistence of regularity}]
Define the map
$$
\Phi_\ep(u_\ep)
:=
S_\ep(t)u_{\ep,0}-i\lambda\int_0^tS_\ep(t-t_1)(|u_\ep|^2u_\ep)(t_1)\, dt_1.
$$
We will show that $\Phi$ is contractive. By Lemma~\ref{lem: Strichartz}, we obtain
\begin{equation}\label{ineq: contraction bound 1}
\|\Phi_\ep(u_\ep)\|_{X_T^s}
\leq
c_1\|u_{\ep,0}\|_{H_x^s}
+c_1\|\la\pa_x\ra^s(|u_\ep|^2u_\ep)\|_{L_T^{\frac{8}{7}}L_x^{\frac{4}{3}}}
\leq
cR_s+cT^{\frac{1}{2}}\|u_\ep\|_{X_T^0}^2\|u_\ep\|_{X_T^s},  
\end{equation}
for some constants $c_1, c>1$, since the fractional Leibniz rule implies that
\begin{equation}\label{ineq: lwp proof step 1}
\|\la\pa_x\ra^s(|u_\ep|^2u_\ep)\|_{L_T^{\frac{8}{7}}L_x^{\frac{4}{3}}}
\ls
\big\|\|u_\ep\|_{L_x^4}^2\big\|_{L_T^4}\|\la\pa_x\ra^su_\ep\|_{L_T^{\frac{8}{5}}L_x^4}
\ls T^{\frac{1}{2}}\|u_\ep\|_{L_T^8L_x^4}^2\|u_\ep\|_{L_T^8W_x^{s,4}}.
\end{equation}
For the difference, using the following inequality
\begin{equation}\label{ineq: cubic difference}
\big||u_1|^2u_1-|u_2|^2u_2\big|\ls (|u_1|^2+|u_2|^2)|u_1-u_2|,
\end{equation}
we can show the following estimate from the similar computation:
\begin{equation}\label{ineq: contraction bound 2}
\|\Phi_\ep(u_\ep)-\Phi_\ep(v_\ep)\|_{X_T^0}
\leq
cT^{\frac{1}{2}}(\|u_\ep\|_{X_T^0}^2+\|v_\ep\|_{X_T^0}^2)\|u_\ep-v_\ep\|_{X_T^0}.   
\end{equation}
Now we define
$$
\mathcal{B}_c:=\{u\in X_T^s : \|u\|_{X_T^0}\leq 2cR_0, \quad \|u\|_{X_T^s}\leq 2cR_s\},
$$
equipped with the metric $d(u_1,u_2):=\|u_1-u_2\|_{X_T^0}$, so that $(\mathcal{B}_c, d)$ is a complete metric space (for example, see \cite[Theorem~1.2.5. and Proof of Theorem~4.4.1.]{Cazenave2003}). 
Take
\begin{equation}\label{eq: T choice}
T:=\frac{1}{(16c^3R_0^2)^2}.
\end{equation}
It follows from \eqref{ineq: contraction bound 1} and \eqref{ineq: contraction bound 2} that $\Phi_\ep$ maps $\mathcal{B}_c$ into itself and is a contraction on $(\mathcal{B}_c, d)$. Hence,  by the Banach fixed point theorem, there exists a unique solution $u_\ep\in X_T^s$ such that $\|u_\ep\|_{X_T^s}\leq 2cR_s.$
\end{proof}
\begin{remark}
We note that the time $T=T(R_0)$ chosen in \eqref{eq: T choice} is independent of $R_s$. Indeed, the contraction is performed in $X_T^0$, so the existence interval is determined by the lower-regularity bound $R_0$, while \eqref{ineq: contraction bound 1} provides the $X_T^s$-bound on the same interval. Thus, the local existence interval determined at the $L^2$-level remains valid at the $H^s$-level. This uniformity of the existence time is used below to derive the exponential $H^s$-bound and to iterate the local-in-time convergence estimate.
\end{remark}

\begin{corollary}[Exponential bound of the $H^s$--norm]\label{cor: Hs norm growth}
Let $u_\ep(t)$ (resp., $u(t)$) be the $H^s(\mathbb{R})$-solution to the biharmonic NLS \eqref{eq: biharmonic NLS} (resp., the NLS \eqref{eq: NLS}), and suppose that
$$
\|u_0\|_{L_x^2} \leq R_0
\qquad \textup{and} \quad \sup_{\ep\in (0,1]}\|u_{\ep,0}\|_{L_x^2}\leq R_0,
$$
and
$$
\|u_0\|_{H_x^s} \leq R_s
\qquad \textup{and} \qquad
\sup_{\ep\in(0,1]}\|u_{\ep,0}\|_{H_x^s}\leq R_s.
$$
Then, there exist constants $K_1, K_2>0$, depending only on $s, R_0$ such that
\begin{equation}\label{ineq: Hs norm growth}
\|u_\ep(t)\|_{H_x^s}+\|u(t)\|_{H_x^s}
\leq 
K_1e^{K_2|t|}
R_s.
\end{equation}
\end{corollary}
\begin{proof}
By the mass conservation law \eqref{eq: L2 cons law}, for any $t\in\R$, we obtain
$$
\sup_{\ep\in(0,1]}\|u_\ep(t)\|_{L_x^2}\leq R_0\quad\textup{and} \quad \|u(t)\|_{L_x^2}\leq R_0.
$$
Hence, by Proposition~\ref{prop: persistence of regularity} and Remark~\ref{rmk: NLS LWP}, there exist $\delta, C_s>0$, depending only on $R_0$ such that
\begin{equation}\label{ineq: norm growth one step}
\|u_\ep(t +\delta)\|_{H_x^s} \leq C_s\|u_\ep(t)\|_{H_x^s} \qquad \textup{and} \qquad
\|u(t+\delta)\|_{H_x^s} \leq C_s\|u(t)\|_{H_x^s}, 
\end{equation}
for all $t\in\R$. Then iterating \eqref{ineq: norm growth one step} yields \eqref{ineq: Hs norm growth}.
\end{proof}

\section{Proof of the main theorem}\label{sec: main thm proof}

\begin{proposition}[Local-in-time convergence]\label{prop: local-in-time convergence}
Let $0\leq s \leq 4$ and suppose that
\begin{equation}\label{ineq: R0, Rs definition}
\sup_{\ep\in(0,1]}\|u_{\ep,0}\|_{L_x^2}, \quad \|u_0\|_{L_x^2}\leq R_0
\qquad
\textup{and}
\qquad
\sup_{\ep\in(0,1]}\|u_{\ep,0}\|_{H_x^s}, \quad \|u_0\|_{H_x^s}\leq R_s.
\end{equation}
For each $\ep\in(0,1]$, let $u_\ep\in C(\R;H_x^s)$ (resp. $u\in C(\R;H_x^s)$) denote the unique solution to the biharmonic NLS \eqref{eq: biharmonic NLS} (resp. the NLS \eqref{eq: NLS}). Then, there exist a time $T=T(R_0)>0$ and a constant $C=C(s, R_0)>1$ which are independent of $\ep$ and $R_s$, such that
$$
\|u_\ep-u\|_{C_t([-T,T];L_x^2)}
\leq
C\|u_{\ep,0}-u_0\|_{L_x^2}
+C\ep^{\frac{s}{2}}R_s.
$$
\end{proposition}

\begin{proof}
Let $T:=\min\{1, c_0^2R_0^{-4}\}$ with sufficiently small $0<c_0\ll1$. Then, by Proposition~\ref{prop: persistence of regularity} and Remark~\ref{rmk: NLS LWP}, $u_\ep(t)$ and $u(t)$ exist on $[-T,T]$ and satisfy
\begin{equation}\label{ineq: uniform bounds of solutions}
\|u_\ep\|_{X_T^0}+\|u\|_{X_T^0}\ls R_0
\qquad \textup{and} \qquad
\|u_\ep\|_{X_T^s}+\|u\|_{X_T^s}\ls R_s.
\end{equation}
Using Duhamel's formula, we expand
$$
\begin{aligned}
u_\ep(t)-u(t)
=&
S_\ep(t)(u_{\ep,0}-u_0)
+(S_\ep(t)-S(t))u_0
-i\lambda
\int_0^tS_\ep(t-t')(|u_\ep|^2u_\ep-|u|^2u)(t')\, dt' \\
&-i\lambda
\int_0^t(S_\ep(t-t')-S(t-t'))(|u|^2u)(t')\, dt'
=:\textup{(I)} + \textup{(II)} + \textup{(III)} + \textup{(IV)}.
\end{aligned}
$$
First, \eqref{ineq: Strichartz 1} and \eqref{ineq: freq loc diff est 1} imply that
$
\|\textup{(I)}\|_{X_T^0}\ls \|u_{\ep,0}-u_0\|_{L_x^2}
$
and
$$
\begin{aligned}
\|\textup{(II)}\|_{X_T^0}
&\leq
\|(S_\ep(t)-S(t))P_{\leq N_0}u_0\|_{X_T^0}
+
\|S_\ep(t)P_{>N_0}u_0\|_{X_T^0}
+
\|S(t)P_{>N_0}u_0\|_{X_T^0}\\
&\ls
\ep^{\frac{s}{2}}T\|u_0\|_{H_x^s}+\|P_{>N_0}u_0\|_{L_x^2}
\ls \ep^{\frac{s}{2}}T\|u_0\|_{H_x^s}+N_0^{-s}\|u_0\|_{H_x^s}
\ls \ep^{\frac{s}{2}}R_s.
\end{aligned}
$$
For the last inequality, we used the fact that $T\leq 1$ and $N_0=\frac{1}{2}\ep^{-\frac{1}{2}}$.

Next, using \eqref{ineq: Strichartz 2} with the inequality \eqref{ineq: cubic difference}, we obtain
$$
\begin{aligned}
\|\textup{(III)}\|_{X_T^0}
\ls
(\|u_\ep\|_{L_T^8L_x^4}^2+\|u\|_{L_T^8L_x^4}^2)T^{\frac{1}{2}}\|u_\ep-u\|_{L_T^8L_x^4}
\ls T^{\frac{1}{2}}R_0^2\|u_\ep-u\|_{L_T^8L_x^4}\ls c_0\|u_\ep-u\|_{X_T^0},
\end{aligned}
$$
where in the last inequality, we used the definition of $T$ and \eqref{ineq: uniform bounds of solutions}.


Similarly, for $\textup{(IV)}$, decompose $|u|^2u=P_{\leq N_0}(|u|^2u)+P_{>N_0}(|u|^2u)$ and applying \eqref{ineq: freq loc diff est 2} and \eqref{ineq: Strichartz 2} respectively, we have
$$
\begin{aligned}
\|\textup{(IV)}\|_{X_T^0}
\ls
\ep^{\frac{s}{2}}T\||u|^2u\|_{L_T^{\frac{8}{7}}W_x^{s,\frac{4}{3}}}
+N_0^{-s}\|P_{>N_0}(|u|^2u)\|_{L_T^{\frac{8}{7}}W_x^{s,\frac{4}{3}}}
\ls
\ep^{\frac{s}{2}}\||u|^2u\|_{L_T^{\frac{8}{7}}W_x^{s,\frac{4}{3}}}.
\end{aligned}
$$
Then by \eqref{ineq: lwp proof step 1} and the definition of $T$ with \eqref{ineq: uniform bounds of solutions}, we get
$$
\|\textup{(IV)}\|_{X_T^0}
\ls
\ep^{\frac{s}{2}}T^{\frac{1}{2}}\|u\|_{X_T^0}^2\|u\|_{X_T^s}
\ls
\ep^{\frac{s}{2}}T^{\frac{1}{2}}R_0^2R_s \ls \ep^{\frac{s}{2}}R_s.
$$
Combining the estimates for $\textup{(I)--(IV)}$, we obtain
$$
\|u_\ep-u\|_{X_T^0}
\ls
\|u_{\ep,0}-u_0\|_{L_x^2}+c_0\|u_\ep-u\|_{X_T^0}+\ep^{\frac{s}{2}}R_s.
$$
Since $c_0>0$ is sufficiently small, our proposition follows.
\end{proof}

Combining the local-in-time convergence and the persistence of regularity, we prove our main theorem.

\begin{proof}[Proof of Theorem \ref{thm: main theorem}]
It suffices to treat $t>0$. Let $R_0, R_s$ be chosen so that \eqref{ineq: R0, Rs definition} holds. By \eqref{ineq: norm growth one step} and Proposition~\ref{prop: local-in-time convergence} with the conservation law \eqref{eq: L2 cons law}, there exist $\delta=\delta(R_0)>0$ and $C=C(s,R_0)$ such that for every $\tau\in\R$
\begin{equation}\label{ineq: Hs bound one step}
\sup_{t\in[\tau, \tau+\delta]}
\Big(\|u_\ep(t)\|_{H_x^s}+\|u(t)\|_{H_x^s}\Big)
\leq
C\Big(\|u_\ep(\tau)\|_{H_x^s}+\|u(\tau)\|_{H_x^s}\Big)
\end{equation}
and
\begin{equation}\label{ineq: local conv one step}
\|u_\ep-u\|_{C_t([\tau,\tau+\delta];L_x^2)}
\leq
C\|u_\ep(\tau)-u(\tau)\|_{L_x^2}
+
C\ep^{\frac{s}{2}}\Big(\|u_\ep(\tau)\|_{H_x^s}+\|u(\tau)\|_{H_x^s}\Big).
\end{equation}
For any $n\in\Z_{\geq 0}$, let $t_n:=n\delta$. Then iterating \eqref{ineq: Hs bound one step} leads to
\begin{equation}\label{ineq: persistence of regularity n step}
\|u_\ep(t_{n-1})\|_{H_x^s}+\|u(t_{n-1})\|_{H_x^s}
\leq
C^{n-1}(\|u_{\ep,0}\|_{H_x^s}+\|u_0\|_{H_x^s})
\leq
2C^{n-1}R_s.
\end{equation}
On the other hand, applying \eqref{ineq: local conv one step} at $\tau=t_{n-1}$ and using \eqref{ineq: persistence of regularity n step}, we obtain
$$
\|u_\ep-u\|_{C_t([t_{n-1},t_n];L_x^2)}
\leq
C\|u_\ep(t_{n-1})-u(t_{n-1})\|_{L_x^2}
+
2C^nR_s\ep^{\frac{s}{2}}.
$$
Therefore, by induction and $2n\leq 2^n$, we obtain
$$
\begin{aligned}
\|u_\ep-u\|_{C_t([0,t_n];L_x^2)}
\leq
C^n\|u_{\ep,0}-u_0\|_{L_x^2}+2nC^nR_s\ep^{\frac{s}{2}}
\leq
(2C)^n\Big(\|u_{\ep,0}-u_0\|_{L_x^2}+R_s\ep^{\frac{s}{2}}\Big).
\end{aligned}
$$
Taking a constant $K(R_0)=\ln(2C)/\delta$. Since $(2C)^n\leq e^{K|t_n|}$, our theorem follows.
\end{proof}

\bibliographystyle{abbrv}
\bibliography{Reference}

\end{document}